\documentclass[11pt]{article}
\usepackage{amsfonts,latexsym,graphics,amssymb}
\usepackage{amsmath,amsthm,amstext,url, enumerate}
\usepackage{graphicx}
\usepackage[margin=1 in]{geometry}
\usepackage{caption}
\usepackage{subcaption}
\usepackage{lineno}
\usepackage{algpseudocode}
\usepackage{algorithm}
\usepackage{color}
\usepackage{epsfig}
\usepackage{tikz}
\usepackage{pstricks}
\usepackage{pst-node}

\title{Many disjoint edges in topological graphs}
\date{}
\author{Andres J. Ruiz-Vargas \\ \'Ecole Polytechnique F\'ed\'erale de Lausanne
        \and Richard Row, \LaTeX\ Academy}
        \author{Andres J. Ruiz-Vargas%
\thanks{Research partially supported by Swiss National Science Foundation grant 200021-137574 and Swiss National
Science Foundation grant 200020-144531 and by Hungarian Science Foundation EuroGIGA Grant OTKA NN 102029.}}%

\newtheorem{theorem}{Theorem}
\newtheorem{claim}[theorem]{Claim}

\newtheorem{lemma}[theorem]{Lemma}

\newtheorem{observation}[theorem]{Observation}
\theoremstyle{definition}
\newtheorem{definition}[theorem]{Definition}

\newtheorem*{remark}{Remark}

\newcommand {\below} {\uparrow}
\newcommand {\ab} {\downarrow}

\newcommand{\comment}[1]{}

\begin{document}
\author{Andres J. Ruiz-Vargas \thanks{Research partially supported by Swiss National Science Foundation grant 200021-137574 and Swiss National
Science Foundation grant 200020-144531 and by Hungarian Science Foundation EuroGIGA Grant OTKA NN 102029.} \\ \'Ecole Polytechnique F\'ed\'erale de Lausanne }
\maketitle

\begin{abstract}
A \emph{monotone cylindrical} graph is a topological graph drawn on an open cylinder with an infinite vertical axis satisfying the condition that every vertical line intersects every edge at most once. It is called \emph{simple} if any pair of its edges have at most one point in common: an endpoint or a point at which they properly cross. We say that two edges are \emph{disjoint} if they do not intersect. We show that every simple complete monotone cylindrical graph on $n$ vertices contains $\Omega(n^{1-\epsilon})$ pairwise disjoint edges for any $\epsilon>0$. As a consequence, we show that every simple complete topological graph (drawn in the plane) with $n$ vertices contains  $\Omega(n^{\frac 12-\epsilon})$ pairwise disjoint edges for any $\epsilon>0$. This improves the previous lower bound of $\Omega(n^\frac 13)$ by Suk which was reproved by Fulek and Ruiz-Vargas. We remark that our proof implies a polynomial time algorithm for finding this set of pairwise disjoint edges.

\end{abstract}

\newpage

\section{Introduction}
A \emph{topological graph} is a graph drawn on the plane so that its vertices are represented by points and its edges are represented by Jordan arcs connecting the respective endpoints. Moreover, in topological graphs we do not allow overlapping edges or edges passing through a vertex.  A \emph{topological graph} is \emph{simple} if every pair of its edges meet at most once, either in a common vertex or at a proper crossing.
We use the words ``vertex'' and ``edge'' in both contexts, when referring to the elements of an abstract graph and also when referring to their corresponding drawings.
A graph is \emph{complete} if there is an edge between every pair of vertices. We say that two edges are \emph{disjoint} if they do not intersect. Throughout this note $n$ denotes the number of vertices in a  graph.

By applying a theorem of Erd\H os and Hajnal \cite{Erdoshajnal}, every complete $n$-vertex simple topological graph contains $e^{\Omega(\sqrt{\log n})}$ edges that are either pairwise disjoint or pairwise crossing. However, it was thought \cite{pach2003unavoidable} that this bound is far from optimal. Fox and Pach \cite{fox2008coloring} showed that there exists a constant $\delta>0$ such that every complete $n$-vertex
simple topological graph contains $\Omega(n^\delta)$ pairwise crossing edges. In 2003, Pach, Solymosi, and T\'oth \cite{pach2003unavoidable} showed that every complete $n$-vertex simple topological graph has at least $\Omega(\log^{1/8}n)$ pairwise disjoint edges. This lower bound was later improved by Pach and T\'oth \cite{PT05} to $\Omega(\log n/ \log \log n)$. Fox and Sudakov \cite{FS09} improved this to $\Omega(\log^{1+\epsilon} n)$, where $\epsilon$ is a very small constant. Furthermore, the previous two bounds hold for dense simple topological graphs. Pach and T\'oth conjectured (see Problem 5, Chapter 9.5 in \cite{BMP05}) that there exists a constant $\delta>0$ such that every complete $n$-vertex simple topological graph has at least $\Omega(n^\delta)$ pairwise disjoint edges. Using the existence of a perfect matching with a low stabbing number
for set systems with polynomially bounded dual shattered function~\cite{CW89}, Suk \cite{S12} settled this conjecture by showing that there are always at least $\Omega(n^{1/3})$ pairwise disjoint edges. This was later reproved by Fulek and Ruiz-Vargas \cite{socgfrv} using completely different techniques.  We are now able to improve the lower bound to $\Omega(n^{1/2-\epsilon})$ for any $\epsilon>0$.

\begin{theorem}\label{main}
A complete simple topological graph on $n$ vertices, which is drawn on the plane,  contains $\Omega(n^{\frac 12-\epsilon})$
pairwise disjoint edges.
\end{theorem}

To the best of our knowledge, no sub-linear upper bound is known for this problem. 


\paragraph{Algorithmic aspects}
Theorem \ref{main} gives a lower bound on the size of a largest independent set in the intersection graph
of edges in a complete simple topological graph.
Besides the fact that computing the maximum number of pairwise disjoint elements in an arrangement of geometric objects is
an old problem in computational geometry, this line of research
is also motivated by applications, e.g., in frequency assignment~\cite{ELR04}, computational cartography~\cite{APS98}, or VLSI design~\cite{HM85}.
Determining the size of a largest independent set is NP-hard already
for intersection graphs of sets of segments in the plane lying in two directions~\cite{KN90}, disks~\cite{FPT81} and rectangles~\cite{AI83}.  For most known cases, efficient algorithms searching for a large independent set in intersection graphs of geometric objects
can only approximate the maximum. It is for this reason that a lot of research has been developed for finding such approximations (see \cite{fox2011computing} for more references).

Our proof of Theorem~\ref{main} yields an efficient algorithm for finding $\Omega(n^{\frac 12-\epsilon})$ pairwise disjoint edges
in a complete simple topological graph. 
Because the number of pairwise disjoint edges is at most $\lfloor n/2 \rfloor$ then our algorithm is a $O(n^{\frac 12+\epsilon} )$-approximation algorithm
for the problem of finding the maximal set of pairwise disjoint edges
in a complete simple topological graph, which beats the factor of $n^{1-\delta}$ for any small $\delta>0$ in the inapproximability  result for the independence number in general graphs due to Zuckerman~\cite{Z07}.

\vspace{1cm}
In Section 2, we introduce cylindrical graphs and state the necessary results showing that in order to prove Theorem \ref{main} it is enough to show that we can always find $\Omega(n^{1-\epsilon})$ pairwise disjoint edges in every complete monotone simple cylindrical  graph. The latter is proved in two steps:  in Section 3, for flags (see Section 3 for the definition); in Section 4, for all graphs.

\section{Cylindrical drawings of graphs}

 Let $\mathcal C$ be the surface of an infinite open cylinder. Formally, $\mathcal C=S^1\times \mathbb R$. We may assume that $S^1$ is the interval $[0,1]$ after gluing the point $0$ to the point $1$. Then,
we can characterise each point $p\in \mathcal C$ by its coordinates: we use $p_x$ to denote its $x$-coordinate and $p_y$ to denote its $y$-coordinate with $0\leq p_x<1$ and $p_y\in \mathbb R$.  
A \emph{cylindrical graph} is a graph drawn on $\mathcal C$ so that its vertices are represented by points and its edges are represented by Jordan arcs connecting the respective endpoints (this is similar to topological graphs with the only difference that the latter are drawn in the plane). As in topological graphs, cylindrical graphs do not allow overlapping edges or edges passing through a vertex.  A \emph{cylindrical graph} is \emph{simple} if every pair of its edges meet at most once either in a common vertex or at a proper crossing.

 The straight lines of $\mathcal C$ with fixed $x$-coordinate will be called the vertical lines of $\mathcal C$. We let $l_{x=a}$ denote the vertical line with $x$-coordinate equal to $a$.
We say that a curve $\gamma\in \mathcal C$ is \emph{x-monotone} if every vertical line intersects $\gamma$ in at most one point. We say that a cylindrical graph is \emph{monotone} if each of its edges is an $x$-monotone curve and furthermore no pair of vertices have the same $x$-coordinate. We will assume that the vertices of any cylindrical graph have $x$-coordinate distinct from zero. 

A monotone cylindrical graph $G$ can be easily represented on the plane: we simply cut $\mathcal C$ along the line $l_{x=0}$ and embed the resulting surface $[0,1)\times \mathbb R$ in the plane in the natural way. We call this the \emph{plane representation} of $\mathcal C$. For a graph $G$ drawn on $\mathcal C$ we will say that the \emph{plane representation} of $G$ is the drawing of $G$ given by the plane representation of $\mathcal C$, note that some edges of $G$ might be cut into two connected components while some edges will remain intact, but with both cases, the edges will stay $x$-monotone curves consisting of either one or two connected pieces, see Figure \ref{cylindricalgraphs}. Throughout this paper we will refer to the plane representation of $\mathcal C$ and $G$ rather than to the actual drawings on $\mathcal C$. Hence we also refer to left and right, so for example we say a point $p\in \mathcal C$ lies to the left (and right) of a point $q\in \mathcal C$ if $p_x<q_x$ ($p_x>q_x$). 


\begin{figure}
\begin{center}
\includegraphics[width=1\columnwidth]{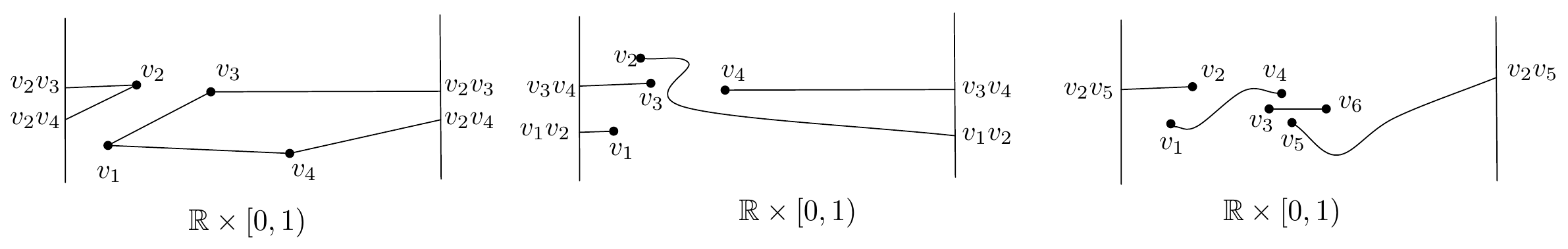}
\end{center}
\caption{Three different plane representations of cylindrical graphs. }
\label{cylindricalgraphs}
\end{figure}
  It is easy to see that a complete topological graph drawn on the plane with its edges being $x$-monotone curves always contains $\lfloor \frac n2 \rfloor$ pairwise disjoint edges. 
In \cite{socgfrv} it was conjectured that a similar statement is true for complete monotone simple cylindrical graphs, that is that every complete  monotone simple cylindrical graph contains at least $cn$ pairwise disjoint edges, for some fixed constant $c$. We prove this conjecture up to a factor of $n^\epsilon$ for any $\epsilon>0$. 
 
\begin{theorem}\label{main2}
A complete monotone simple cylindrical graph on $n$ vertices contains $\Omega (n^{1-\epsilon})$ pairwise disjoint edges for any $\epsilon>0$. 
\end{theorem}

Theorem \ref{main2} is the main contribution of this paper. Theorem \ref{main} follows by combining Theorem \ref{main2} with the following Theorem of R. Fulek. 

\begin{theorem}\label{mainlemma}\cite{radoreduction}
Let $c(n')$ denote the maximum number $c$ such that every complete monotone simple cylindrical graph on $n'$ vertices contains a disjoint matching of size at least $c$. For every complete simple topological graph $G$ with $n$ vertices there exists $\Delta=\Delta(G)$, $0 < \Delta < n$, such that $G$ contains a set of at least $\max \{\frac{n}{\Delta},c(\Delta)\}$ pairwise disjoint edges.
\end{theorem}

\subsection*{Proof of Theorem \ref{main}}
\begin{proof} 
 Let $\epsilon>0$ and $G$ be a simple topological graph with $n$ vertices. By Theorem \ref{main2} it follows that every complete monotone simple cylindrical graph with $n'$ vertices has a set of $cn'^{1-\epsilon}$ pairwise disjoint edges for some constant $c(\epsilon)$. By Theorem \ref{mainlemma} there exists a $\Delta=\Delta(G)$ such that $G$ has a set of at least $\max \{\frac{n}{\Delta}, c\Delta^{1-\epsilon}\}$ pairwise disjoint edges. If $\Delta\leq n^{1/2}$ then $n^{1/2-\epsilon}\leq n^{1/2} \leq n/\Delta \leq \max \{\frac{n}{\Delta}, c\Delta^{1-\epsilon}\}$. If, on the other hand, $\Delta>n^{1/2}$ then 
\[ cn^{1/2-\epsilon}\leq c(n^{1/2})^{1-\epsilon}\leq c\Delta^{1-\epsilon}\leq \max \left \{\frac{n}{\Delta}, c\Delta^{1-\epsilon}\right\}. \]
Hence in both cases, we obtain a set with $\Omega(n^{\frac 12-\epsilon})$ pairwise disjoint edges.
\end{proof}

Although we do not present an explicit algorithm for Theorem \ref{main2}, proofs are inductive and can easily be changed to a polynomial time algorithm. Furthermore, the application of Theorem \ref{mainlemma} can also be done in polynomial time \cite{radoreduction}. Hence we give an implicit polynomial time algorithm for finding a set of $\Omega({n^{1/2-\epsilon}})$ pairwise disjoint edges in every simple complete topological graph. 
In the following sections we prove Theorem \ref{main2}.

\subsection{Notation and Definitions}
The notations in this section  will be used extensively during this paper to refer to topological properties of cylindrical graphs .  

We say that two graphs drawn on the same surface are \emph{disjoint} if there is no pair of edges, one from each graph, that are intersecting. For the rest of the paper all considered curves (and hence all considered edges) will be $x$-monotone in $\mathcal C$. Let $G$ be a complete monotone simple cylindrical graph and $e=(vu),f =(wz)$ be two edges of $G$ which do not share vertices. We will say that $f$ \emph{has endpoints to the left of} $e$ if
\[ w_x,z_x< v_x,u_x.\]

\subsection*{Relating edges}
\begin{definition}
Let $e,f$ be two $x$-monotone curves in $\mathcal C$. 
We say that $e,f$ are \emph{related} if  $e$ and $f$ are non-crossing, there is a vertical line intersecting the relative interior of both $e$ and $f$ and furthermore for every vertical line intersecting the relative interior of both $e$ and $f$ the order of intersection is the same. \end{definition}
Given two related curves $e,f$ we write $e\below f$ if for every vertical line $l$ that intersects the relative interior of both $e$ and $f$ we have $(l\cap e)_y\leq  (l \cap f)_y$ (we define $e\ab f$ analogously).  If $e$ and $f$ are edges of a cylindrical graph that are related then either $e\below f$ or $e\ab f$ but both cannot be true. Note that two non-crossing edges are not necessarily related, even if there is a vertical line that intersects the relative interior of both of them, as one can see in the middle Figure \ref{cylindricalgraphs} that $v_1v_2$ and $v_3v_4$ are not related. However if two edges share an endpoint and there is a vertical line intersecting the relative interior of both of them, then they are necessarily related.

Given a point $v\in \mathcal C$ and an $x$-monotone curve $e$ we say that $e$ and $v$ are \emph{related} if the vertical line $l_{x=v_x}$ intersects $e$ in its relative interior. If $v$ and $e$ are related then we say that $v$ is \emph{below} $e$ if $v_y< (l_{x=v_x}\cap e)_y$, and we denote this as $v\below e$ and $e\ab v$ (we define $v$ to be \emph{above} $e$, and denote it as $v\ab e$ and $e \below v$, analogously). 

\textbf{Remark.}  Note that $\below$ is defined only between two curves or a point and a curve, and furthermore $\below$ does not give a proper order to the edges: for example one may see in the rightmost of Figure \ref{cylindricalgraphs} that $v_2v_5\below v_3v_6$, $v_3v_6\below v_1v_4$ and $v_1v_4 \below v_2v_5$. Note that if $\{a,b,c\}$ is a set of three edges such that every pair is related, and furthermore there is a vertical line $l$ intersecting all of the elements of $\{a,b,c\}$ then it does follow that $a\below b$ and $b\below c$ implies that $a\below c$. 



  \textbf{Remark.} Throughout this paper we will assume that no two edges cross in a point with $x$-coordinate equal
 to zero, and no vertex has $x$-coordinate equal to zero.

\section{Flags}
For complete simple topological graphs drawn in the plane with $x$-monotone curves it is easy to find a set of $\lfloor n/2 \rfloor $ pairwise disjoint edges: order the vertices from left to right as $v_1,...,v_n$ (i.e., with ascending $x$-coordinate), the edges $\{ v_{2k+1}v_{2k +2}: k\in \{0,...,\lfloor n/2\rfloor -1\} \}$ are pairwise disjoint.  In some cases the plane representation of a complete simple cylindrical graph is also a complete simple topological graph drawn on the plane using $x$-monotone curves (this is the case when no edge crosses $l_{x=0}$), in this case finding a set of $\lfloor n/2 \rfloor $ pairwise disjoint edges is easy.

 We say that a complete monotone simple cylindrical graph $G$ is a \emph{flag} if $l_{x=0}$ intersects all edges of $G$ in their relative interior. 
 Note
 that whether a complete simple cylindrical graph is a flag depends on the set of coordinates used, in
 particular, a graph might not be a flag even if there is a vertical line that intersects all edges
 in their interior, however such a graph would become a flag by a suitable change of coordinates. Our first ingredient for proving Theorem \ref{main2} is to prove it for flags. 
\begin{theorem}\label{flag}
A complete monotone simple cylindrical graph that is a flag with $n\geq 10$ vertices contains a set with $\lceil n/25\rceil +1$ pairwise disjoint edges.
\end{theorem}

In fact,  we will prove something stronger than Theorem \ref{flag}, that is we will find a very specific set of disjoint edges. This extra information will be crucial in the proof of Theorem \ref{main2} (without this extra properties Claim \ref{inductionstep} in the next section would not be true).  
 \begin{definition}
We say that a set of pairwise disjoint edges $\mathcal E$ in a flag $G$ is \emph{proper} if for any two edges  $e,f\in \mathcal E$, $e$ and $f$ are related, and furthermore if $e \below f$ then either
\\(a) $e$ has both endpoints to the left of both endpoints of $f$
\\(b) or $f$ has both endpoints to the left of both endpoints of $e$
\\(c) or there is an edge $g\in E(G)$ with both endpoints to the left of the endpoints of $e$ and $f$ and such that the endpoints of $e$ are below $g$ while the endpoints of $f$ are above $g$ (in other words $g$ separates the endpoints of $e$ and $f$). 

\end{definition}
\begin{theorem}\label{flagstrong}
A complete monotone cylindrical graph that is a flag with $n\geq 10$ vertices contains a proper set of $\lceil n/25\rceil +1$ pairwise disjoint edges.
\end{theorem}

The idea of the proof of Theorem \ref{flagstrong} is simple. We process the vertices by increasing $x$-coordinate, at each iteration we delete some vertices, and we either find a way of separating the remaining vertices into two parts $V_1$ and $V_2$ for which the corresponding induced graphs are disjoint, or we find an edge $e$ and a set of vertices $V'$ such that $G[V']$ is disjoint with $e$. The ``$+1$'' in the statement of the theorem is quite important, it gives us space to delete some vertices at each iteration and apply induction on two separate pieces (see details in the next subsection). 

\subsection{Basic observations}
We now prove some properties for flags, subsequently we put them together to prove Theorem \ref{flagstrong}. 
We say that an edge $e$ is \emph{circular} if $l_{x=0}$ intersects $e$ in its interior. Hence a monotone simple cylindrical graph is a flag if all of its edges are circular. Every circular edge $e=vw$ with $v_x< w_x$ is naturally partitioned into two curves by the line $l_{x=0}$: the one consisting of the points in the curve with coordinate smaller than $v_x$ will be called the \emph{negative part} of $e$ and will be denoted $e^-=(vw)^-$;  the one consisting of the points of $e$ with coordinate greater than $w_x$ will be called the \emph{positive part} and will be denoted as $e^+=(vw)^+$, see Figure \ref{transitivity}.

\subsubsection*{More definitions: extensions of edges.}
Given two curves $e$ and $f$ we write $e-f$ to denote the set of points of $e$ that are not in $f$. 
Let $G$ be a simple monotone cylindrical graph and $e$ be an edge of $G$. We say that $e'$ is an \emph{extension} of $e$ if $e'$ is a closed curve such that $e\subseteq e'$ and furthermore $e'$ is an $x$-monotone curve in $\mathcal{C}$ (see Figure \ref{transitivity}).

\begin{figure}
\centering
\includegraphics{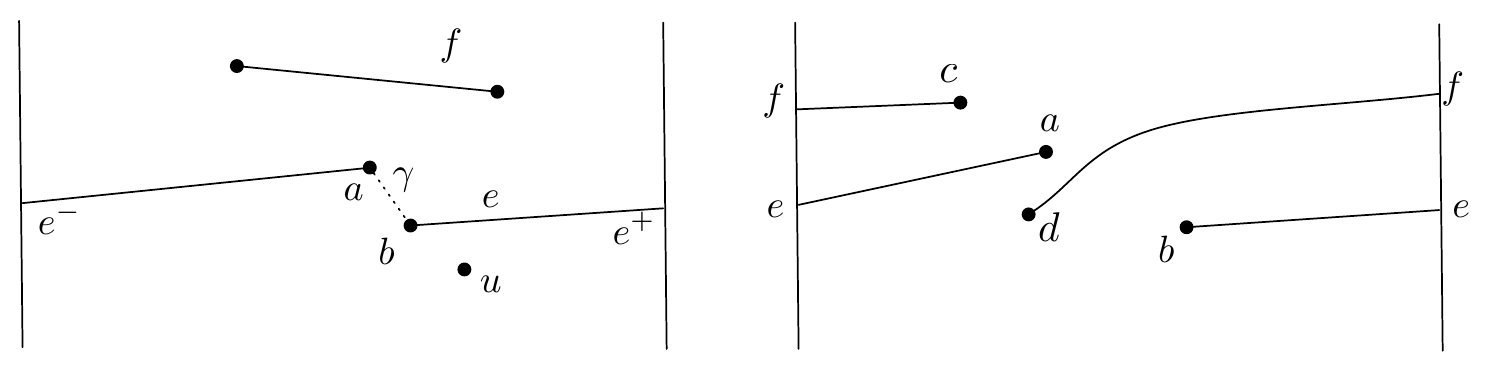}
\caption{Left: A circular edge $e$ together with its division into $e^-$ and $e^+$. $e\cup \gamma$ is an extension of $e$. Right: two circular edges which are not related, in this case $b \below f$ and $f\below a$.}
\label{transitivity}
\end{figure}

\begin{observation}\label{paircrossing}
Let $e$ be an $x$-monotone curve in $\mathcal C$. Let $a,b$ be two points in $\mathcal C$ with $a,b\below e$. Let $\gamma$ be an $x$-monotone curve joining $a$ and $b$ such that any vertical line that intersects $\gamma$ intersects $e$ as well. Then $e$ and $\gamma$ cross an even number of times. Furthermore, if $e$ and $\gamma$ do not cross then it follows that $\gamma \below e$. 
 \end{observation}
 In the case that $e$ and $\gamma$ in the observation are edges on a simple cylindrical graph (or even only parts of edges), then it is clear that $\gamma$ and $e$ do not cross. 

\begin{observation}\label{aboutrelated}
Let $G$ be a simple monotone cylindrical graph and $e=ab$ and $f=cd$ be two disjoint circular edges of $G$. Then either $e$ and $f$ are related or $a\below f$ and $f \below b$ or $b\below f$ and $f \below a$ (see Figure \ref{transitivity}). 
\end{observation}
Observation \ref{aboutrelated} is a case analysis. Note that it is only true for $e$ and $f$ circular edges. In short the observation says that either the edges are related, or the vertices of one edge are on different sides (with respect to the $y$-coordinate) of the other edge, and viceversa.  So for example, it implies that if $e=ab$ and $f=cd$ are two disjoint circular edges and $a,b\below f$ then $e$ and $f$ are related. 

\begin{observation}
Given two related curves $e$ and $f$ if $e \below f$ then:
\\(1)There is an extension $f'$ of $f$ such that $f'$ and $e$ are related and $e\below f'$.
\\(2) There is an extension $e'$ of $e$ such that $e'$ and $f$ are related and $e'\below f$. 
\end{observation}

\begin{observation}\label{trivial} 
Let $G$ be a flag and let $a,b,c\in V(G)$. Then the edge $ab$ and the edge $ac$ are related. Furthermore if
 $c$ and $ab$ are related, then $c\below ab$ if and only if $ac\below ab$. Hence, $c\ab ab$ if and only if $ac\ab ab$
\end{observation}
\textbf{Remark.} Note that in the above observation, there is no specification on the order of the vertices $a,b,c$. It is easy to see that two edges that share a vertex are related as long as there is a vertical line that intersects the interior of both. Since $G$ is a flag then $l_{x=0}$ is such a line. Hence, if $c$ is below $ab$ then there is a point of $ac$ that is below $ab$ and hence all points of $ac$ that are comparable with $ab$ must indeed be below $ab$.

 \subsection{Properties of Flags}
 For this subsection let $G$ be a flag, that is a complete simple monotone cylindrical graph such that $l_{x=0}$ intersects all edges of $G$ in its interior, and let $v_1,...,v_n$  be the vertices of $G$ ordered by increasing $x$-coordinate.

\begin{claim}\label{impossiblecrossing}
Let $0<i_1<i_2<i_3<i_4\leq n$ such that $v_{i_3},v_{i_4}\below v_{i_1}v_{i_2}$ ($v_{i_3},v_{i_4}\ab v_{i_1}v_{i_2}$). The following three properties are equivalent (See Figure \ref{wrongintersection}): 
\begin{itemize}
\item 1. $v_{i_1}v_{i_2}$ and $v_{i_3}v_{i_4}$ cross
\item 2. $(v_{i_1}v_{i_2})^+$ and $(v_{i_3}v_{i_4})^-$ cross (Figure \ref{impossiblea})
\item 3. $v_{i_2}\below v_{i_3}v_{i_4}, v_{i_1}\ab v_{i_3}v_{i_4}$ and $(v_{i_3}v_{i_4})^-\below (v_{i_1}v_{i_2})^-$ ( $v_{i_2}\ab v_{i_3}v_{i_4}, v_{i_1}\below v_{i_3}v_{i_4}$ and $(v_{i_3}v_{i_4})^-\ab (v_{i_1}v_{i_2})^-$)
\end{itemize}
Furthermore $v_{i_1}v_{i_2}$ and $v_{i_3}v_{i_4}$ are disjoint if and only if $v_{i_1},v_{i_2}\ab v_{i_3}v_{i_4}$ ($v_{i_1},v_{i_2}\below v_{i_3}v_{i_4}$). 
\end{claim}

\begin{proof}
Note that the last statement is trivial, it is easy to see that $v_{i_1}v_{i_2}$ and $v_{i_3}v_{i_4}$ do not cross if and only if $v_{i_1},v_{i_2}\ab v_{i_3}v_{i_4}$ (as in Figure \ref{impossibled}). Then let us show the 3 equivalencies. 

It is clear that 2 implies 1. We now assume 1 and show that 2 holds. Hence assume $v_{i_1}v_{i_2}$ and $v_{i_3}v_{i_4}$ cross. Since $v_{i_1}v_{i_2}=(v_{i_1}v_{i_2})^+\cup (v_{i_1}v_{i_2})^-$ and $v_{i_3}v_{i_4}=(v_{i_3}v_{i_4})^+\cup (v_{i_3}v_{i_4})^-$ then in order to show that 2 holds it it enough to show that out of the four possible ways these two edges could cross, three are impossible.  Since $ i_2<i_3$ then it cannot be that $(v_{i_1}v_{i_2})^-$ and $(v_{i_3}v_{i_4})^+$ cross
 . Hence it remains to check that it cannot happen that $(v_{i_1}v_{i_2})^-$ and $(v_{i_3}v_{i_4})^-$ cross or that $(v_{i_1}v_{i_2})^+$ and $(v_{i_3}v_{i_4})^+$ cross. 
 
Assume that $(v_{i_1}v_{i_2})^+$ and $(v_{i_3}v_{i_4})^+$  cross (see Figure \ref{impossibleb}). Hence $v_{i_1}\below v_{i_3}v_{i_4}$ and $v_{i_2}\ab v_{i_3}v_{i_4}$. It is clear that $v_{i_2}$ is related to $v_{i_3}v_{i_4}$ and hence by Observation \ref{trivial} it follows that $v_{i_2}v_{i_4}\ab v_{i_3}v_{i_4}$. Also, $v_{i_1}$ and $v_{i_2}v_{i_4}$ are related, and since $v_{i_1}\below v_{i_3}v_{i_4}$ it follows that $v_{i_1}\below v_{i_2}v_{i_4}$. On the other hand, since $v_{i_4}\below v_{i_1}v_{i_2}$, then by Observation \ref{trivial} it follows that $v_{i_2}v_{i_4}\below v_{i_1}v_{i_2}$ and hence $v_{i_2}v_{i_4} \below v_{i_1}$. Therefore $v_{i_1}\below v_{i_2}v_{i_4} \below v_{i_1}$ which by definition means $(v_{i_1})_y< (l_{x=(v_{i_1})_x}\cap v_{i_2}v_{i_4})_y< (v_{i_1})_y$,  a contradiction.  Note that the same argument works in the case that $(v_{i_1}v_{i_2})^-$ and $(v_{i_3}v_{i_4})^-$cross. 

 If $(v_{i_1}v_{i_2})^+$ and $(v_{i_3}v_{i_4})^-$ cross then the edges must look like in the Figure \ref{impossiblea} and hence $v_{i_2}\below v_{i_3}v_{i_4}, v_{i_1}\ab v_{i_3}v_{i_4}$ and $(v_{i_3}v_{i_4})^-\below (v_{i_1}v_{i_2})^-$. Analogously if $v_{i_2}\below v_{i_3}v_{i_4}$ then it is clear that $v_{i_1}v_{i_2}$ and $v_{i_3}v_{i_4}$ must cross, and the equivalencies are proved.
\end{proof}

\begin{figure}
        \centering
        \begin{subfigure}[b]{0.4\textwidth}
                \includegraphics[width=\textwidth]{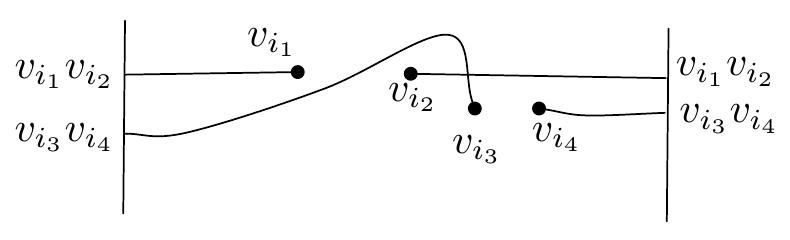}
                \caption{$(v_{i_1}v_{i_2})^+$ and $(v_{i_3}v_{i_4})^-$ cross}
                \label{impossiblea}
        \end{subfigure}%
        ~ 
        \begin{subfigure}[b]{0.4\textwidth}
                \includegraphics[width=\textwidth]{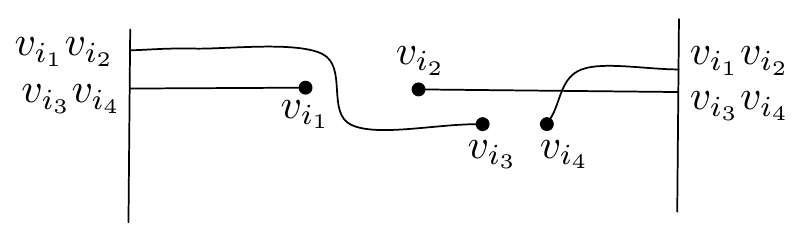}
                \caption{$(v_{i_1}v_{i_2})^+$ and $(v_{i_3}v_{i_4})^+$ cross}
                \label{impossibleb}
        \end{subfigure}
        ~ 
        \begin{subfigure}[b]{0.4\textwidth}
                \includegraphics[width=\textwidth]{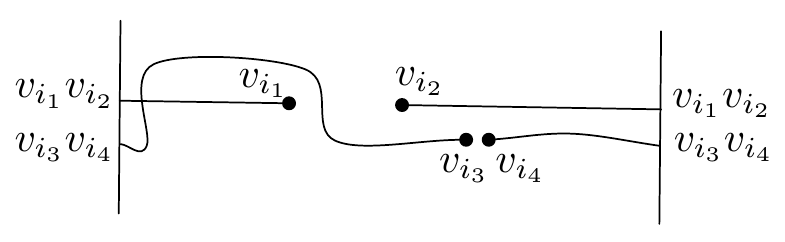}
                \caption{$(v_{i_1}v_{i_2})^-$ and $(v_{i_3}v_{i_4})^-$ cross}
                \label{impossiblec}
        \end{subfigure}
                \begin{subfigure}[b]{0.4\textwidth}
                \includegraphics[width=\textwidth]{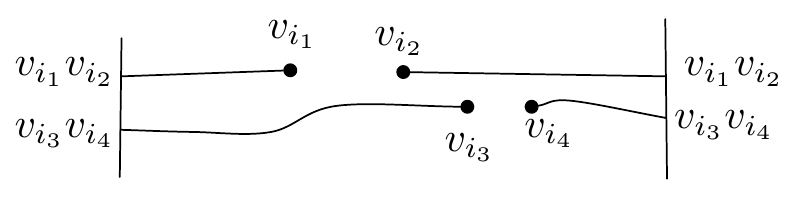}
                \caption{$v_{i_1}v_{i_2}$ and $v_{i_3}v_{i_4}$ are disjoint}
                \label{impossibled}
        \end{subfigure}
        \caption{The different drawings of two edges $v_{i_1}v_{i_2}$ and $v_{i_3}v_{i_4}$ when $v_{i_3},v_{i_4}\below v_{i_1}v_{i_2}$}\label{wrongintersection}
\end{figure}

\begin{lemma}\label{niceorder}
Let $V^+=\{v_i: v_i\ab v_1v_2\}$ then for $v_{i_1},v_{i_2}\in V^+$ with $i_1<i_2$ we have $v_1v_{i_1}\below v_1v_{i_2}$. Analogously, let $V^-=\{v_i: v_i\below v_1v_2\}$ then for $v_{i_1},v_{i_2}\in V^-$ with $i_1<i_2$ we have $v_1v_{i_1}\ab v_1v_{i_2}$ (see Figure \ref{corollary10}).
\end{lemma}

\begin{proof}
The proof of the two different cases is exactly the same. For a contradiction let $v_{i_1},v_{i_2}\in V^-$ with $i_1<i_2$ and such that $v_1v_{i_1}\below v_1v_{i_2}$ (see Figure \ref{lemma9}). Note that since  $v_{i_1},v_{i_2}\in V^-$ it follows from Observation \ref{trivial} that $v_1v_2 \ab v_1v_{i_1}$ and $v_1v_2 \ab v_1v_{i_2}$.
It is clear that $v_{i_2}\ab v_1v_{i_1}$ and hence, by Observation \ref{trivial}, it follows that $v_{i_1}v_{i_2}\ab v_1v_{i_1}$. Since $v_1$ is related with $v_{i_1}v_{i_2}$ it follows that $v_{i_1}v_{i_2}\ab v_1$. On the other hand, by Claim \ref{impossiblecrossing}, it follows that regardless whether $v_1v_2$ and $v_{i_1}v_{i_2}$ cross or not $v_{i_1}v_{i_2}\below v_1$, a contradiction. 

\end{proof}

\begin{figure}
        \centering
        \begin{subfigure}[b]{0.4\textwidth}
                \includegraphics[width=\textwidth]{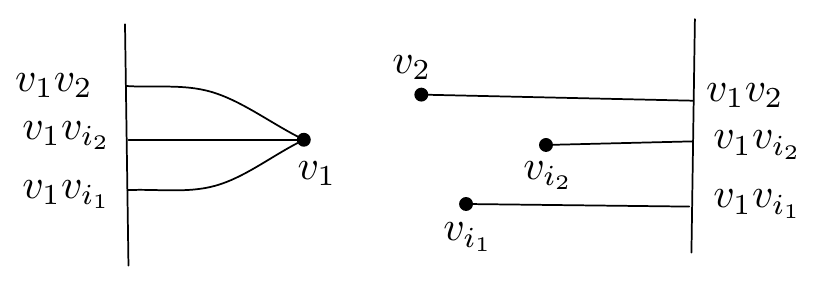}
                \caption{$v_1v_{i_1}\below v_1v_{i_2}$ as in the proof of Lemma \ref{niceorder}}
                \label{lemma9}
        \end{subfigure}%
        ~ 
        \begin{subfigure}[b]{0.4\textwidth}
                \includegraphics[width=\textwidth]{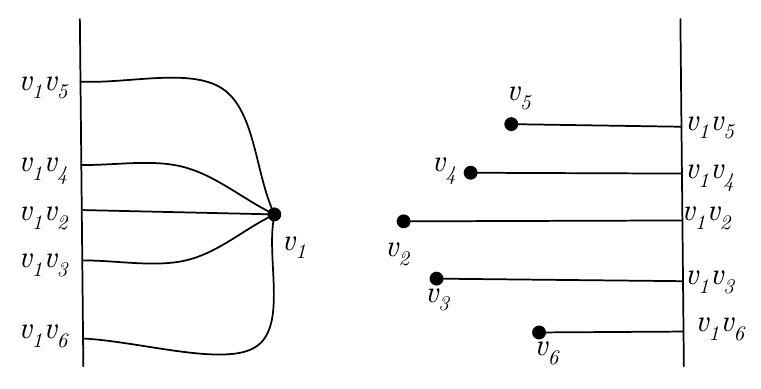}
                \caption{Lemma \ref{niceorder}}
                \label{corollary10}
        \end{subfigure}
        ~ 
    \label{pyramid}
    \caption{Figures for Lemma \ref{niceorder}.}
\end{figure}

For two vertices, $v_iv_j$, $i<j$ let $V^+(v_iv_j)=\{v_s: s>j, v_s\ab v_iv_j\}$ and $V^-(v_iv_j)=\{v_s: s>j, v_s\below v_iv_j\}$. Note that in the definition of $V^+(v_iv_j)$ ($V^-(v_iv_j)$) we only consider the vertices above (below) $v_iv_j$ that have $x$-coordinate bigger than $v_j$. This is because on each subsequent iteration we will only consider vertices with bigger $x$-coordinate.  We say that $v_iv_j$ is \emph{separating} if $|V^+(v_iv_j)|>1$ and $|V^-(v_iv_j)|>1$, see Figure \ref{separating}. 
 
 \begin{figure}
\begin{center}
\includegraphics[width=1\columnwidth]{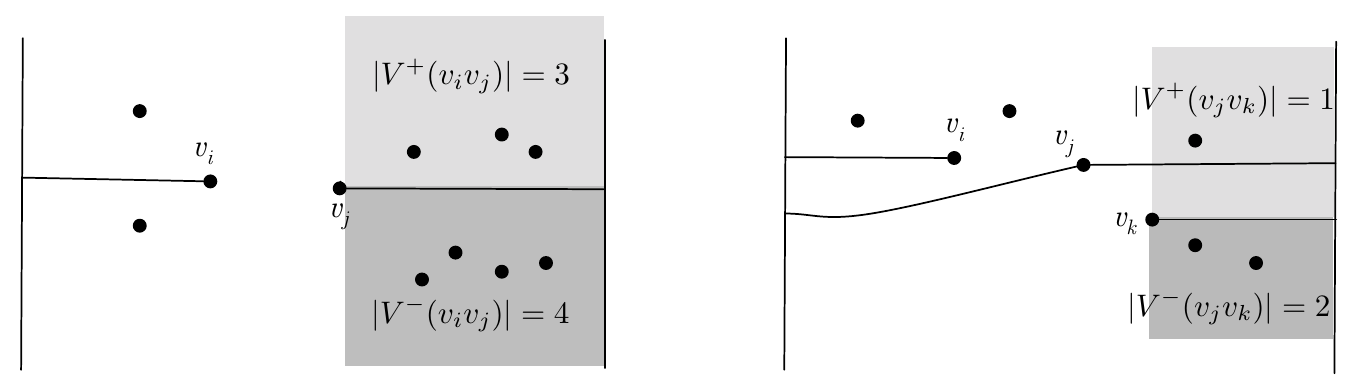}
\end{center}
\caption{A separating edge $v_iv_j$ on the left, a good upper triplet $(v_i,v_j,v_k)$ on the right. }
\label{separating}
\end{figure}

The following three claims will form the core of the proof of Theorem \ref{flagstrong}. 
 \begin{claim}\label{separatingedge}
If $v_iv_j$ is a separating edge then $G[V^+(v_iv_j)]$ and $G[V^-(v_iv_j)]$ are disjoint. Furthermore if $e$ is an edge in $G[V^+(v_iv_j)]$ and $f$ is an edge in $G[V^-(v_iv_j)]$ then $e\ab f$. 
\end{claim}

\begin{proof}
 Refer to Figure \ref{inductiona}. 
  Let   
 $u$ denote the vertex in $V^+(v_iv_j)$ with smallest $x$-coordinate and, $w$ the vertex in $V^-(v_iv_j)$ with smallest $x$ coordinate. We first show that the edge $v_iu$ is disjoint from all the edges in $G[V^-(v_iv_j)]$ and that $v_iw$ is disjoint from all the edges in $G[V^+(v_iv_j)]$. 
 
 The cases are analogous hence we just show that the edges in $G[V^-(v_iv_j)]$ are disjoint from $v_iu$. For a contradiction let $e=ab$ be an edge with endpoints in $V^-(v_iv_j)$ such that $e$ and $v_iu$ cross. Note that from Observation \ref{trivial} it follows that $v_iv_j\below v_iu$.  By Claim \ref{impossiblecrossing} applied to $v_i,v_j,a,b$ it follows that $e\below v_i$. It follows that the crossing of $e$ and $v_iu$ must be between $(v_iu)^+$ and $e^-$ (see definitions at the beginning of this section) and hence $u_x\leq a_x,b_x$.  By Observation \ref{trivial} it follows that $v_ju \ab v_iv_j$ and hence $a,b\below v_ju$. Therefore, by Claim \ref{impossiblecrossing} applied to $v_j,u,a,b$ it follows that $ab\below v_j$. Again, by Claim \ref{impossiblecrossing} the latter means that $ab$ and $v_iv_j$ don't cross from which it follows, because  $v_iv_j\below v_iu$, that $e$ and $v_iu$ do not cross, a contradiction.


Now to prove the first part of the claim, 
let $e=ab$ be an edge with endpoints in $V^-(v_iv_j)$ and $f=cd$ be an edge with endpoints in $V^+(v_iv_j)$.  Assume that $c$ has smaller $x$-coordinate then $d$. Either $c=u$ or by Lemma \ref{niceorder} and Observation \ref{trivial} it follows that $c\ab v_iu$. Regardless of $c$, by Lemma \ref{niceorder} $d\ab v_iu$ and $v_iu\ab v_iv_j$. Hence, it is clear that we can draw an extension $f'$ of $f$ such that $f'-f\ab v_iu$. Similarly, assume $a$ has smaller $x$-coordinate than $b$ then either $a=w$ or $a\below v_iw$. Furthermore, $b\below v_iw$ and $v_iw\below v_iv_j$ and we can draw an extension $e'$ of $e$ such that $e'-e\below v_iw$.
Let $\theta$ be a curves from $v_i$ to $u$ such that $\theta\cup v_iu$ is an extension of $v_iu$ and furthermore $v_iv_j\below (\theta\cup v_iu)$. If $e,f$ cross then it follows that since $e'$ and $f'$ are closed curves then they cross an even number of times, hence $e'$ and $f'$ cross at least twice. Since $e$ and $f$ can cross only once and since it is clear that $e'-e$ does not cross $f'-f$ it follows that either $e$ crosses $f'-f$ or $f$ crosses $e'-e$. Assume first that $e$ crosses $f'-f$.  Since $f-f'$ is above $\theta\cup v_iu$ it follows that $e$ must cross $\theta\cup v_iu$. Furthermore, for $e$ to cross $f-f'$ it follows that $e$ must cross $v_iu$, otherwise no point of $e$ lies above $v_iu$. This is a contradiction, as we have seen that $e$ and $v_iu$ do not cross.  Similarly it follows that $f$ and $e'-e$ cannot cross. Hence $e$ and $f$ do not cross.

From the previous paragraph it follows that $e'$ and $f'$ do not cross this, together with the fact that $e$ and $f$ are circular implies that $e$ and $f$ are related. Hence $e \below f$ (i.e. $f\ab e$). 
\end{proof}

 For $i<j<k$ if $v_jv_k\below v_iv_j$ and $|V^+(v_jv_k)|\leq 1$ then we say that $(v_i,v_j,v_k)$ form a \emph{good upper triplet} of vertices. Similarly if $v_jv_k\ab v_iv_j$ and $|V^-(v_jv_k)|\leq 1$ then we say that $v_i,v_j,v_k$ form a \emph{good lower triplet} of vertices, see Figure \ref{separating}. 
 
 \begin{claim}\label{triplet}
If $(v_i,v_j,v_k)$ forms a good upper (lower) triplet of vertices then $G[V^-(v_jv_k)]$ ($G[V^+(v_jv_k)]$) is disjoint of $v_iv_j$. Furthermore, $v_iv_j$ is related to all edges of $G[V^-(v_jv_k)]$ ($G[V^+(v_jv_k)]$).
\end{claim} 
\begin{proof}

Refer to Figure \ref{inductionb}. Suppose that $(v_i,v_j,v_k)$ form a good upper triplet an let $e=v_{i_1}v_{i_2}$ in $G[V^-(v_jv_k)]$ with $i_1<i_2$. Because $v_{i_1},v_{i_2}\below v_jv_k$ by Claim \ref{impossiblecrossing} it follows that $v_{i_1}v_{i_2}\below v_j$. Hence, by Claim \ref{impossiblecrossing}, $v_{i_1}v_{i_2}$ does not cross $v_iv_j$.

\end{proof}


\begin{claim}\label{inductionlemma}
Among $\{v_1,v_2,v_3,v_4,v_5,v_6\}$ there is either a separating edge, a good upper triplet of vertices or a good lower triplet of vertices.
\end{claim}
\begin{proof}
Assume for a contradiction that the claim is not true. Let $z_1=v_1$ and $z_2=v_2$. Then it follows that either $|V^+(z_1z_2)| \leq 1$ or $|V^-(z_1z_2)| \leq 1$. Assume, without loss of generality that  $|V^+(z_1z_2)| \leq 1$. Let $z_3$ be the vertex with smallest $x$ coordinate in $V^-(z_1z_2)$, hence it is clear that $z_3 \in\{v_3,v_4\}$. Analogously,  for $z_2z_3$ either $|V^+(z_2z_3)| \leq 1$ or $|V^-(z_2z_3)| \leq 1$. By Observation \ref{trivial} $z_2z_3\below z_1z_2$ and hence if $|V^+(z_2z_3)|\leq 1$ then $z_1,z_2,z_3$ is a good upper triplet. So let us assume that  $|V^-(z_2z_3)| \leq 1$. 

Let $z_4$ be the vertex with smallest $x$-coordinate in $V^+(z_2z_3)\setminus V^+(z_1z_2)$ (note that $V^+(z_2z_3)\setminus V^+(z_1z_2)\subseteq V^-(z_1z_2)$).  Since $z_3\in\{v_3,v_4\}$ it follows that $z_4 \in\{v_4,v_5,v_6\}$.  Analogously to the previous analysis we know that if $|V^-(z_3z_4)| \leq 1$ then $z_2,z_3,z_4$ is a good lower triplet. Hence we may assume that $|V^+(z_3z_4)| \leq 1$. We claim that $z_1,z_3,z_4$ is a good upper triplet. From Lemma \ref{niceorder} since $z_3,z_4\leq z_1z_2$ it follows that $z_1z_3\ab z_1z_4$ and hence $z_4\below z_1z_3$. It then follows by Observation \ref{trivial} that $z_3z_4\below z_1z_3$. Hence $z_1,z_3,z_4$ is a good upper triplet.


  \end{proof}

\begin{figure}
        \centering
        \begin{subfigure}[b]{0.4\textwidth}
                \includegraphics[width=\textwidth]{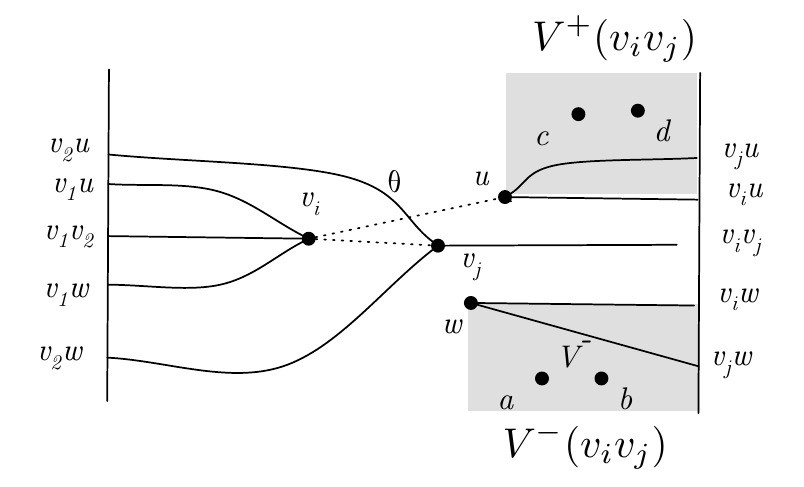}
                \caption{Drawing for Claim \ref{separatingedge}.}
                \label{inductiona}
        \end{subfigure}%
        ~ 
        \begin{subfigure}[b]{0.4\textwidth}
                \includegraphics[width=\textwidth]{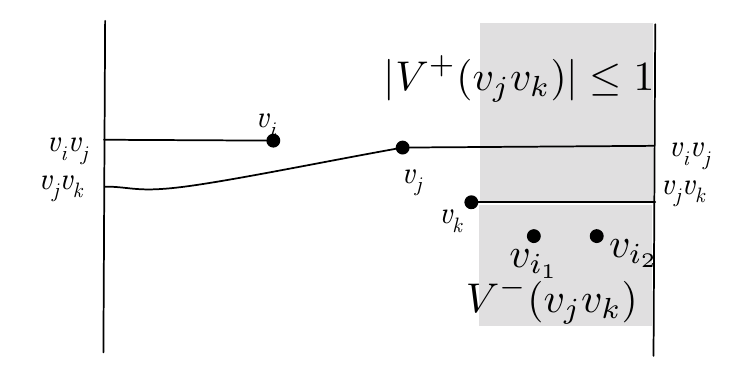}
                \caption{Drawing for Claim \ref{triplet}}
                \label{inductionb}
        \end{subfigure}
        ~ 
        \caption{ }
    \label{inductionfig}
\end{figure}

\textbf{Remark.} We note that for the proofs of the previous three claims all of the properties of a flag are used: that is the fact that a flag is a complete monotone simple cylindrical graph for which the line $l_{x=0}$ intersects the relative interior of all of its edges. The removal of any of these properties would render these claims false.

\subsection{Proof of Theorem \ref{flagstrong}}
 The proof of Theorem \ref{flagstrong} follows directly by an induction that combines Claims \ref{separatingedge}, \ref{triplet} and \ref{inductionlemma}. Some care needs to be done for the base case of the induction. 

\begin{proof}

 We will prove the theorem by induction on the number of vertices. For the base case let us prove the theorem for  $V(G)=10$, hence we need to prove that there are 2 disjoint edges. 
 We first use Claim \ref{inductionlemma} and hence we either have a separating edge, a good upper triplet or a good lower triplet among the six vertices with smallest $x$-coordinate.
 
First we will assume that we have a separating edge $v_i,v_j$ among the vertices $v_1,v_2,v_3,v_4,v_5,v_6$. Recall that this means that $|V^+(v_iv_j)|>1$ and $|V^-(v_iv_j)|>1$. By Claim \ref{separatingedge} all edges in $G[V^+(v_iv_j)]$ are disjoint from all of the edges in $G[V^-(v_iv_j)]$. 
Hence we can pick any edge in $G[V^-(v_iv_j)]$ and any edge in $G[V^+(v_iv_j)]$ and furthermore, by Claim \ref{separatingedge} this pair of edges is indeed proper set of disjoint edges since $v_iv_j$ has both endpoints to the left of the endpoints of both of the edges just picked.

Now assume that $v_i,v_j,v_k$ is a good upper triplet, by  Claim \ref{triplet} no edge in $ G[V^-(v_jv_k)]$ crosses $v_iv_j$ and since $|V^+(v_iv_j)|\leq 1$ it follows that there must be at least $10-6-1=3$ vertices in $ G[V^-(v_jv_k)]$. Hence any edge in $ G[V^-(v_jv_k)]$ together with $v_iv_j$ form a set of 2 pairwise disjoint edges, furthermore $v_iv_j$ has endpoints to the left of any edge in  $ G[V^-(v_jv_k)]$, it follows from the latter that the two edges are related and hence the set of 2 disjoint edges is proper. 

The case when $v_i,v_j,v_k$ is a good lower triplet is analogous to the latter, and the theorem is true for any graph with 10 vertices.. 

 Having proved the base case, we note that this takes care of any graph with at most 25 vertices: for these graphs we must only find a proper set of 2 pairwise disjoint edges. 
Hence, assume that $|V(G)|>25$. Now let us prove the induction step. 

Again we use Claim \ref{inductionlemma}. So assume $v_iv_j$ is a separating edge with endpoints in $v_1,...,v_6$, then, 
by Claim \ref{separatingedge}, no edge of $G[V^+(v_iv_j)]$ crosses an edge of $G[V^-(v_iv_j)]$. 
By our inductions hypotheses if both $V^-(v_iv_j)$ and $V^+(v_iv_j)$ have both at least 10 vertices then we can find a set with 

 \[ \left \lceil \dfrac{|V^+(v_iv_j)|}{25}\right \rceil +1+\left \lceil \dfrac{|V^-(v_iv_j)|}{25}\right \rceil +1 \geq \left \lceil \dfrac{|V^+(v_iv_j)|+|V^-(v_iv_j)|}{25} \right \rceil +2=\left \lceil\dfrac{n-6}{25}\right \rceil +2\geq  \left \lceil\dfrac{n}{25}\right \rceil +1\]
pairwise disjoint edges.

 Now let us analyse the case when one of the sets $V^+(v_iv_j)$ and $V^-(v_iv_j)$ has 9 vertices or less. Assume, without loss of generality, that $V^+(v_iv_j)$ has 9 vertices or less. Since $|V(G)|>25$ it follows that $V^-(v_iv_j)$ has at least 11 vertices.  Take any edge $e$ with both endpoints in $V^+(v_iv_j)$, then by Claim \ref{separatingedge} $e$ is disjoint of all edges in $G[V^-(v_iv_j)]$  and hence using induction we can find at least
 \[ 1+  \left \lceil \dfrac{|V^-(v_iv_j)|}{25}\right \rceil +1\geq\left  \lceil\dfrac{n-15}{25}\right \rceil +2\geq  \left \lceil\dfrac{n}{25}\right \rceil +1\]
 pairwise disjoint edges. 

It is easy to see that the set of disjoint edges that we have found is proper. Namely by our induction hypotheses the sets of edges found in $G[V^+(v_iv_j)]$ and $G[V^+(v_iv_j)]$ are proper. Note that given an edge $e$ in $G[V^-(v_iv_j)]$ and an edge $f$ in $G[V^+(v_iv_j)]$ then $e\below f$ by Claim \ref{separatingedge} and the edge $v_iv_j$ has endpoints to the left of both $e$ and $f$, and furthermore, the endpoints of $e$ lie below $v_iv_j$ and the endpoints of $f$ lie above $v_iv_j$. Hence the union of these two proper sets is indeed a proper set of disjoint edges.

Now assume that Claim \ref{inductionlemma} gives us a good upper triplet $v_i,v_j,v_k$. Then
 by Claim \ref{triplet} no edge in $ G[V^-(v_jv_k)]$ crosses $v_iv_j$ and hence we can apply induction to $ V^-(v_jv_k)$ and add the edge $v_iv_j$ to obtain at least
\[  \left \lceil \dfrac{n-6-1}{25}\right \rceil +1+1\geq \left \lceil\dfrac{n}{25}\right \rceil +1
\]
pairwise disjoint edges. This time $v_iv_j$ has endpoints to the left of any edge from $ G[V^-(v_jv_k)]$. By Claim \ref{triplet} $v_iv_j$ is related to any edge of $ G[V^-(v_jv_k)]$, and hence the set of disjoint edges found is proper. 

The case when Claim \ref{inductionlemma} provides a good lower triplet is similar to the previous case.  The theorem follows. 
\end{proof}

\begin{remark} Note that the proof of Theorem \ref{flagstrong} is also an efficient algorithm to find a set of disjoint edges in a flag: at each iteration we either find a new edge to add to the set of disjoint edges, or split the vertices into two sets into which we may apply induction separately. 
\end{remark}

\section{Proof of Theorem \ref{main2}}

 \begin{figure}
\begin{center}
\includegraphics[width=.7\columnwidth]{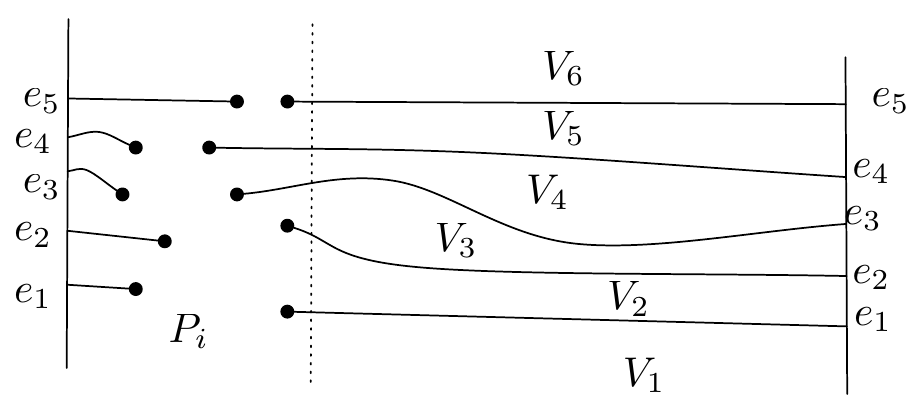}
\end{center}
\caption{The proper set of disjoint edges found on $P_i$, and the partition into $V_i$'s. }
\label{slabs}
\end{figure}

\begin{proof}
Assume $0<\epsilon <\frac 12$ and let $n_0$ be a natural number big enough such that the following two equations hold for all $m\geq n_0$. 

\begin{equation}\label{equ2}
m^{\epsilon^2} \geq 10^6
 \end{equation}

\begin{equation}\label{equ3}
m^\epsilon \geq 10 \lceil \log n_0 \rceil +1 \end{equation}
Let $c=n_0^{\epsilon-1}$ and
 $f(n)=cn^{1-\epsilon}$ . We will show that for any graph with $n$ vertices we can always find a set of at least $\lceil f(n) \rceil$ pairwise disjoint edges. 
 
 We first note some properties of $f$: 
  it is a strictly increasing function with strictly decreasing, positive, first derivative. 
  Hence, for $m>2$ and $0<x<m$ we have
  \begin{equation}\label{aboutf2}
  f(m-x)\geq f(m)-f'(m-x)x.
  \end{equation}
  
  For a fixed $m>2$ let $h(y)=y^{1-\epsilon} +(m-y)^{1-\epsilon}$ then 
$h'(y)=(1-\epsilon)y^{-\epsilon}- (1-\epsilon)(m-y)^{-\epsilon}$. Hence $h'$ is at least zero for all $0< y\leq m/2$. 
It follows that for any $x$ such that $0\leq x<m/2$ and $a,b\geq x$ with $a+b=m$ we have

\begin{equation}\label{aboutf} f(a)+f(b) \geq f(m-x)+f(x). \end{equation}
Finally, note that for $m\geq n_0$, Equation \eqref{equ2} implies that $m^\epsilon\geq 10^6$ and since $\epsilon<1/2$ we have $c=n_0^{\epsilon-1}\leq n_0^{-\epsilon}\leq 10^{-6}$.

We now prove by induction on $n$ that every graph with $n$ vertices has a set of at least $\lceil f(n) \rceil$ pairwise disjoint edges.  For $1<n\leq n_0$ we have $f(n)\leq 1$ and the statement is true. Therefore, assume $n>n_0$  and that the statement holds for all graphs with less than $n$ vertices and let $G$ be a complete monotone simple cylindrical graph with $n$ vertices $v_1,...,v_n$ ordered with increasing $x$-coordinate.

Let $k:=\lfloor \frac{n}{10f(n)}\rfloor$ and note that $k>10$.  Since $n> n_0$ implies $f(n)\geq 1$, we have
\[ \lfloor n/k \rfloor \geq \left \lfloor\frac{n}{n/10f(n)} \right \rfloor  \geq  \lfloor 10f(n) \rfloor \geq 9f(n).\]

Partition $\mathcal C$ with vertical lines such that each part, with the exception of at most one, contains in its interior exactly $k$ vertices, and furthermore no vertex lies in the boundary of any part.
 Let $\mathcal P=\{P_1,...,P_{\lceil n/k\rceil}\}$ be the set consisting of these parts. Let $\mathcal P'$ be the subset of $\mathcal P$  consisting of the sets in $\mathcal P$ which fully contain at least one edge in its interior.  If $|\mathcal P'|\geq f(n)$ we are done.

Therefore, since $\lfloor n/k \rfloor > f(n)$ we may assume that there exists $P_i\in \mathcal P$ with exactly $k$ vertices in its interior, such that no edge is completely contained in $P_i$. Without loss of generality, we may assume that $P_i=\{p\in \mathcal C| 0<p_x<a\}$ for some $a\in (0,1)$, that is $P_i$ is bounded by the vertical lines $l_{x=0}$ and $l_{x=a}$ and $V':=\{v_1,...,v_k\}$ is the set of vertices inside $P_i$. 

Note that all edges in $G[V']$ must intersect $l_{x=0}$, that is $G[V']$ is a flag, hence by Theorem \ref{flagstrong} we can find in $G[V']$ a proper set $\mathcal E$ of $\alpha:=\lceil k/25 \rceil$ pairwise disjoint edges.  Note that since $l_{x=0}$ intersects all edges of $\mathcal E$ then we may order the edges $\mathcal E$ to be $e_1,...,e_\alpha$ in such a way that $e_1\below e_2,...,e_{\alpha-1}\below e_\alpha$.


For every integer $i$ such that $1<i\leq \alpha$ let $V_i$ be the vertices of $V-V'$ that are above $e_{i-1}$ and below $e_i$. Let $V_1$, $V_{\alpha+1}$ be the vertices below $e_1$ and above $e_\alpha$, respectively (see Figure \ref{slabs}). Note that for all $i$, all vertices in $V_i$ have greater $x$ coordinate than all the vertices in $V'$.

\begin{claim}\label{inductionstep}
Let $s$ be an integer such that $1<s<\alpha+1 $. Let $U=\bigcup_{i>s}V_i$ and $W=\bigcup_{i<s}V_i$ then the edges in $G[U]$ are disjoint from the edges in $G[W]$. Furthermore every edge of $G[U]$ is disjoint of $e_{s-1}$ and every edge of $G[W]$ is disjoint of $e_{s}$.  
\end{claim}
We note that for the proof of Claim \ref{inductionstep} it is crucial to use the fact that $G[V']$ is a flag. In fact, if $G[V']$ did not have all the properties of a flag, the claim would not be true. 
The proof of this claim consists of a different case analysis showing that in all cases the specified edges are indeed disjoint. It is in some of the cases that the hypotheses of Theorem \ref{flagstrong} are needed, namely that the set of disjoint edges given by Theorem \ref{flagstrong} is proper. 

Note that in this section we go back to working with simple monotone cylindrical drawings. In other words, the drawings are not necessarily flags. The following observation will be useful during the proof of Claim \ref{inductionstep}, it essentially shows that for some type of circular edges, there are two different ways the edges can cross. 
\begin{observation}\label{basics}
Let $e=v_{i_1}v_{i_2}$ and $f=v_{i_3}v_{i_4}$ be two circular edges of $G$ such that $i_1< i_2< i_3< i_4$ and such that $v_{i_3},v_{i_4}\below e$. Then, if $e$ and $f$ are disjoint then $f\below e$ and hence $v_{i_1}\ab f$ and  $v_{i_2} \ab f$. 
 \\ If $e$ and $f$ cross, then one of the following is true, as can be seen in Figure \ref{wrongintersection}.
\\(1)  $v_{i_1}\ab f$ and $v_{i_2}\below  f$, hence $f^-\below e^-$.
\\(2) $v_{i_1} \below f$ and $v_{i_2} \ab f$, hence $f^-\below e^+$, .
\end{observation}

A similar statement for the previous observation when $v_{i_3},v_{i_4}\ab e$ has been omitted for presentation reasons, but it will be used in this section.

\subsubsection*{Proof of Claim \ref{inductionstep}}

The proof of Claim \ref{inductionstep} is a detailed case analysis that uses Observation \ref{basics} extensively. However, one must draw some edges from $G[V']$ in order to force the corresponding edges to be disjoint. 

Observation \ref{paircrossing},
Observation \ref{trivial} and Observation \ref{basics} will be used, sometimes without mention, throughout this proof. It will also be widely used that all the vertices in $V-V'$ have greater $x$-coordinate than all of the vertices of $V'$, and that the graph is simple, that is any two edges intersect at most once, at a crossing or at a common endpoint.

We claim that in order to prove the claim it is enough to show that no edge of $G[W]$ intersects $e_s$ and that no edge of $G[U]$ intersects $e_{s-1}$. Indeed, let $e=vu$ be an edge in $G[W]$ crossing an edge $f=wz$ in $G[U]$. Assume first that $e$ and $f$ are both circular. 
Let $e'$ and $f'$ be extensions of $e,f$, respectively, such that $e'-e$ is below $e_{s-1}$ and $f'-f$ is above $e_s$. Then because $e$ and $f$ cross then $e'$ and $f'$ must cross an even number of times, but clearly $e'-e$ and $f'-f$ are disjoint, and hence either $e$ crosses $f'-f$ or $f$ crosses $e'-e$. It follows that either $e$ must cross $e_s$ or that $f$ must cross $e_{s-1}$. 
Note that if one of the edges from $e$ and $f$ is not circular, say $e$, then it follows with the same reasoning as above that $f$ must cross $e_{s-1}$. If neither of the edges is circular it follows that $e$ and $f$ are disjoint, since $e$ will be below $e_{s-1}$ and $f$ will be above $e_{s}$.

Let $e$ be an edge in $G[W]$ and $f$ be an edge in $G[U]$, we now show that $e$ does not cross $e_s$ and that $f$ does not cross $e_{s-1}$. We assume that both $e$ and $f$ are circular, for otherwise it is easy to see (for example, by Observation \ref{paircrossing}) that $e$ does not cross $e_{s-1}$ and $f$ does not cross $e_s$.

Let $a,b$ denote the endpoints of $e_{s-1}$ and $c,d$ the endpoints of $e_s$ such that $a_x< b_x$ and $c_x< d_x$ and assume without loss of generality that $a_x< c_x$. Then one of the following cases must be true (refer to Figure \ref{maintheoremfig}):
\\ (1). $a_x<b_x< c_x<d_x$
\\ (2). $a_x< c_x< b_x< d_x$
\\ (3). $a_x< c_x< d_x< b_x$

We break the analysis into these three different cases.

\textbf{(1). Assume $a_x<b_x< c_x<d_x$:}

By Observatin \ref{trivial} $bc\ab ab$ and $bc \below cd$.\\
We now show that $e$ does not cross $e_s$.
By Observation \ref{basics} it follows that $e^-$ is either below $(bc)^-$ or $(bc)^+$.
If $e^-$ is below $(bc)^-$ then $e^-$ does not cross $(e_{s-1})^+$ and hence $e$ does not cross $bc$ and therefore $e$ does not cross $e_s$. If $e^-$ is below $(bc)^+$ it also follows that $e$ does not cross $e_s$.

We now show that $f$ does not cross $e_{s-1}$.
By Observation \ref{basics} it follows that $f^-$ is either above $(bc)^+$ or $(bc)^-$. 
If it is above $(bc)^+$ then it also must be above $e_{s}^-$ hence $f$ and $e_{s-1}$ do not cross. If it is above $(bc)^-$ it follows, directly, that $f$ does not cross $e_{s-1}$. 

For both (2) and (3) we use the fact that $\mathcal E$ is a set of disjoint edges is proper. In both cases $e_{s-1}$ does not have both endpoints to the left of both endpoints of $e_s$ nor viceversa, then it follows that there is a circular edge $g$ with endpoints in $V'$ such that $a,b\below g$ and $c,d\ab g$ and $g$ has endpoints to the left of both $e_{s-1}$ and $e_s$.  From Claim \ref{impossiblecrossing} it follows that $g$ does not intersect $(e_{s-1})^+$ and $g$ does not intersect $(e_s)^+$. However, $g$ may intersect $e_{s-1}$ or $e_s$, and hence we draw it, in Figure \ref{maintheoremfig}, as a dotted line, and we do not draw one of its endpoints. 
Let $t$ be the vertex of $g$ with smallest $x$-coordinate then it follows because $G[V']$ is a flag that we can apply Claim \ref{impossiblecrossing} and get that $ab\below t\below cd$ .

\textbf{(2). Assume $a_x< c_x< b_x< d_x$:}

 First we will show that $ab\below ac\below cd$.  Because $a\below cd$ it follows that $ac\below cd$. Recall that $a\below g$ hence it follows that $ta\below g$. Since $g\below c,d$ and $ta$ and $c,d$ are related it follows that $ta\below c,d$. This together with the fact that $t,a\below cd$ and that $ab \below t$ implies $ab\below ta\below cd$. Therefore $ac\ab ta$ and hence $ac\ab ab$ (i.e. $ab \below ac$). 

We now show that $e$ does not cross $e_s$.
If $e^-$ is below $(ac)^+$ then $e$ does not cross $e_s$. Assume $e^-$ is below $(ac)^-$, this implies that $e^-$ does not cross $(ta)^+$ and hence $e^-$ is below $(ta)^+$ which implies that $e$ does not cross $e^s$.  

We now show that $f$ does not cross $e_{s-1}$.
If $f^-$ is above $(ac)^+$ then $f$ and $e_{s-1}$ do not cross. Assume that $f^-$ is above $(ac)^-$. Note that $cb$ is below $ac$ and that $cb$ is also below $ab$. If $f^-$ is above $(cb)^-$ then $f^-$ is also above $(cd)^-$ and hence $f$ does not intersect $ac$ and therefore $f$ and $e_{s-1}$ are disjoint. Then assume that $f^-$ is above $(cb)^+$, then it is also above $(ab)^+$ and since we had assumed that $f^-\ab (ac)^-$ then $f^-\ab (ab)^-$. It follows that $f$ and $e_{s-1}$ are disjoint. 

\textbf{(3). Assume $a_x< c_x< d_x< b_x$ :}

 In this case it is easy to see that $ta$ is below $g$ which implies, because $c,d\ab g$ and $t,a\below cd$ that $ta\below cd$. Clearly $ta$ is above $ab$.  Therefore $ac$ is above $ta$ which implies that $ac$ is above $ab$. It is clear that $ac\below cd$. 

We now show that $e$ does not cross $e_s$.
Assume first that $e^-$ is below $e_s^-$. Then $e^-$ must also be below $(ac)^+$ and hence $e$ and $e^s$ are disjoint. 

Hence assume that $e^-$ is below $e_s^+$. If $e^-$ is either below $(ac)^+$ or $(ta)^+$ it follows that $e$ and $e_s$ are disjoint. Note however, that $e^-$ must be either below $(ac)^+$ or $(ta)^+$. Hence $e$ and $e_s$ are disjoint. 

We now show that $f$ does not cross $e_{s-1}$.

Assume that $f^-$ is above $e_{s-1}^-$. It follows that $f^-$ is also above $(ac)^-$ from which follows that it must also be above $(ta)^+$ and hence $f$ and $e_{s-1}$ are disjoint. 

Hence assume $f^-$ is above $e_{s-1}^+$. Note that $f^-$ must be either above $(ac)^+$ or above $(ta)^+$. In both cases it follows that $f$ and $e_{s-1}$ are disjoint.

\begin{figure}
        \centering
        \begin{subfigure}[b]{1\textwidth}
                \includegraphics[width=\textwidth]{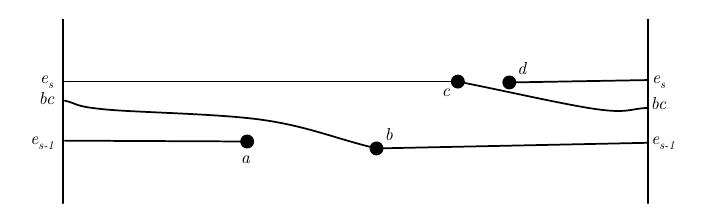}
                \caption{Case (1) in the proof of Claim \ref{inductionstep}}
                \label{maintheoa}
        \end{subfigure}%
        
        ~ 
        \begin{subfigure}[b]{1\textwidth}
                \includegraphics[width=\textwidth]{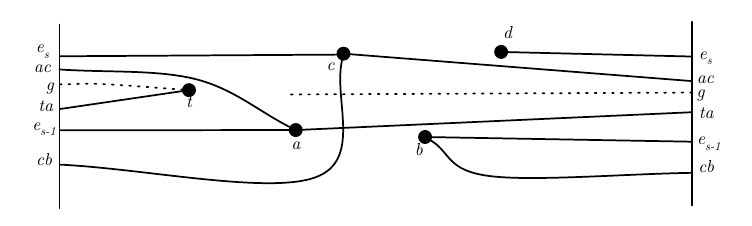}
                \caption{Case (2) in the proof of Claim \ref{inductionstep}}
                \label{maintheob}
        \end{subfigure}
        
        \begin{subfigure}[b]{1\textwidth}
                \includegraphics[width=\textwidth]{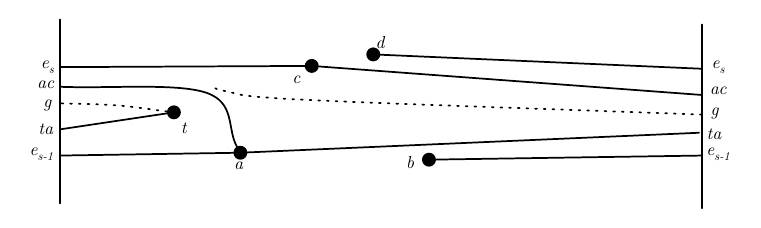}
                \caption{Case (3) in the proof of Claim \ref{inductionstep}}
                \label{maintheoc}
        \end{subfigure}
\caption{Figures for the proof of Claim \ref{inductionstep}}
        ~ 
\label{maintheoremfig}
\end{figure}

This concludes the proof of Claim \ref{inductionstep}.

\begin{claim}\label{verticesbetween} If
 $|V_i\cup V_{i+1}|\leq \frac{n}{10f(n)}$ for some $i$ such that $1\leq i\leq \alpha$, then we can find $f(n)$ pairwise disjoint edges in $G$. 
\end{claim}
For the proof of Claim \ref{verticesbetween} suppose that $|V_i\cup V_{i+1}|\leq \frac{n}{10f(n)}$ for some $i$ such that $1\leq i\leq \alpha$. Let $V^-=\bigcup_{j<i} V_j$ and $V^+=\bigcup_{j>i+1}V_j$. Then, by Claim \ref{inductionstep}, the edge $e_{i}$ is disjoint from all edges in $G[V^-]$ and $G[V^+]$. Furthermore, the edges in $G[V^-]$ are disjoint from the edges in $G[V^+]$. Hence we can find a set of $ f(|V^-|)+f(|V^+|)+1 $ pairwise disjoint edges. Since $f(n)>1$ it follows that $\frac{n}{5f(n)}\leq n/5$, using this and Equation \eqref{aboutf2} we have that
  
\begin{equation*}
\begin{split}
 f(|V^-|)+f(|V^+|)+1  &\geq   f(|V|-k-|V_i|-|V_{i+1}|)+1\geq  f\left (n- \frac{n}{5f(n)}\right )+1\\
 &\geq f(n)-f'\left (n- \frac{n}{5f(n)}\right )\frac{n}{5f(n)}+1\\
 &\geq f(n)-f'(n/2)\frac{n}{5f(n)}+1\\
 &= f(n) -(1-\epsilon)c \frac{2^{\epsilon}}{n^{\epsilon}}\frac{n}{5cn^{1-\epsilon}} +1
 \\ &= f(n)- (1-\epsilon)\frac{ 2^\epsilon}{5} +1 \geq f(n).
 \end{split}
\end{equation*} 



 
This concludes the proof of Claim \ref{verticesbetween}.

Order the $V_i$ by size, i.e., let $j_1,...,j_{\alpha+1}$ be a permutation of $1,...,\alpha+1$ such that $|V_{j_1}|\leq |V_{j_2}|\leq....\leq |V_{j_{\alpha+1}}|$. Let $\mathcal I=\{i| \frac{n}{20f(n)}\leq |V_{j_i}|\leq 10^5f(n)\}$. Note that $\mathcal I$ consists of a set of consecutive integers and recall that $|V_{j_i}|\leq |V_{j_{i+1}}|$. We now show that $|\mathcal I |\geq 10 \lceil \log n\rceil +1$. 

 By Claim \ref{verticesbetween} we may assume that there at least 
 \[ 
 \left \lfloor \frac{\alpha+1}{2}\right \rfloor \geq \frac{\alpha}{2} \geq  \frac{\lfloor n/(10f(n)) \rfloor}{50}\geq  
 \frac{ n/(10f(n)) -1}{50} \] 
 
 sets in 
$\{V_{j_1},...,V_{j_{\alpha +1}} \}$ with more than $\frac{n}{20f(n)}$ vertices. On the other hand the number of these sets having at least $10^5f(n)$ vertices is at most $\frac{n}{10^5f(n)}\geq 1$.  Hence, 
\begin{equation*}
\begin{split}
|\mathcal I| &\geq \left \lfloor \frac{\alpha+1}{2}\right \rfloor -\frac{n}{10^5f(n)}
\geq  \frac{ n/(10f(n)) -1}{50} -\frac{n}{10^5f(n)}
\\ & = \frac{n} {500f(n)} -\frac{n}{10^5f(n)} -\frac{1}{50} \geq \frac{n} {500f(n)} -2\frac{n}{10^5f(n)}
\\&  \geq \frac{n}{1000f(n) }=\frac{n^\epsilon }{1000c }
\\ &\geq 10 \lceil \log n\rceil +1  
\end{split}
\end{equation*}
where the last inequality follows by Equation \eqref{equ3} and the fact that $c\leq 1/10^6$. 

\begin{claim} \label{anl}There is an $l\in \mathcal I$ such that $j_l,j_{l+1},j_{l+2}\in \mathcal I$ and $|V_{j_{l+2}}|\leq 2|V_{j_l}|$.\end{claim}

Assume that the statement is not true. Then $|V_{j_{l+2}}|> 2|V_{j_l}|$ for all $l$ with $l,l+2\in \mathcal I$. That means that the size of every second set increases by at least a factor of two. Let $a$ be the last index in $\mathcal I$. Then

\[ |V_{j_a}|> 2^{\lfloor |\mathcal I|/2 \rfloor} \frac{n}{20f(n)} \geq 2^{4\log n} \frac{n}{20f(n)}=\frac{n^5}{20f(n)}>n\]
a contradiction.   Hence Claim \ref{anl} is proved.

Let $l$ be given as in Claim \ref{anl}, and let $z_1,z_2,z_3$ be the permutation of $j_{l},j_{l+1},j_{l+2}$ such that $z_1<z_2<z_3$. Let $U=\bigcup_{i>z_2} V_i$ and $W=\bigcup_{i<z_2} V_i$ and let $x=|V_{l}|$. Then it follows that $|U|,|W|>x$. On the other hand, we have $|U|+|W|= n-k-|V_{z_2}| \geq n-k-2x$ and $\frac{n}{20f(n)}\leq x\leq 10^5f(n)$. Recall that $k=\lfloor \frac{n}{10f(n)}\rfloor$ hence 
\[|U|+|W|\geq n-k-2x\geq n-4x.\]

By our induction hypotheses we can find a set of $f(|U|)$ pairwise disjoint edges in $G[U]$ and a set of $f(|W|)$ pairwise disjoint edges in $G[W]$.
By Claim \ref{inductionstep} the union of these two sets is a set of $f(|U|)+f(|W|)$ pairwise disjoint edges in $G$. 
Hence, from $|U|,|W|\geq x$, Equation \eqref{aboutf}, Equation \eqref{aboutf2} and since $5x\leq 5\cdot 10^5f(n)\leq \frac{1}{2} n^{1-\epsilon}\leq \frac{n}{2}$ it follows that 
\begin{equation}\label{first}
\begin{split}
f(|U|)+f(|W|) &\geq f(n-5x)+f(x)  \geq  f(n)-f'(n-5x)\cdot5x+f(x)
\\& \geq f(n)-f'(n/2)\cdot 5x +f(x)= f(n)-c(1-\epsilon)\frac{2^{\epsilon}}{n^{\epsilon}} \cdot5x +cx^{1-\epsilon}
\\ & =f(n) -cx^{1-\epsilon} (1-\epsilon)\frac{2^\epsilon}{n^\epsilon}\cdot 5x^\epsilon +cx^{1-\epsilon}
\end{split}
\end{equation} 
Recall that $x\leq 10^5f(n)$ and hence by Equation \eqref{equ2} it follows that 
\[(1-\epsilon)\frac{2^\epsilon}{n^\epsilon}\cdot 5x^\epsilon \leq (1-\epsilon)\frac{2^\epsilon}{n^\epsilon}\cdot 5(10^{5}cn^{1-\epsilon})^\epsilon= (1-\epsilon)\cdot2^\epsilon \cdot 5\cdot 10^{5\epsilon} c^\epsilon n^{-\epsilon ^2}\leq 1.\]
Putting this in Equation \eqref{first} it follows that $f(|U|)+f(|W|)\geq f(n)$ and the theorem is proved. 
\end{proof}

Note that although the proof of Theorem \ref{main2} is a bit long, the algorithmic part is very simple. Essentially one partitions $\mathcal C$ into $\mathcal P$ and either enough parts of $\mathcal P$ contain an edge inside and we  put these edges together to form a set of disjoint edges. Otherwise, there must be a part that induces a flag. Then Theorem \ref{flagstrong} is invoked,  the $V_i$'s are constructed, and a suitable place for induction is found. 

\textbf{Acknowledgments}. I would like to thank J\'anos Pach, Bartosz Walczak, Carsten Moldenhauer and Yuri Faenza for useful comments and discussion.

\bibliographystyle{plain}
\bibliography{refs}


\end{document}